\documentclass[11pt]{amsart}
\usepackage{amsthm}
\usepackage{fullpage}
\usepackage{hyperref}
\usepackage{verbatim}
\usepackage{tikz}
\newtheorem{theorem}{Theorem}
\newtheorem{lemma}[theorem]{Lemma}

\newtheorem*{example}{Example}
\newtheorem*{lemmaB}{Lemma A}
\newtheorem*{theoremB}{Theorem B}
\def\mad{\textrm{mad}}

\def\ceil#1{\lceil#1\rceil}

\begin{document}
\title{Injective colorings of graphs with low average degree}
\author{
Daniel W. Cranston$^{\ast}$  \and 
Seog-Jin Kim$^{\dagger}$ \and 
Gexin Yu$^{\ddagger}$}

\thanks{$^{\ast}$ 
Department of Mathematics \& Applied Mathematics, 
Virginia Commonwealth University, Richmond, VA; DIMACS, Rutgers University, Piscataway, NJ.  Email: \texttt{dcranston@vcu.edu}.}

\thanks{$^{\dagger}$ Konkuk University, Seoul, Korea. Email: \texttt{skim12@konkuk.ac.kr}; corresponding author.  Research supported by the Korea Research Foundation Grant funded by the Korean Government (KRF-2008-313-C00115 )}

\thanks{$^{\ddagger}$ College of William and Mary, Williamsburg, VA.
Email: \texttt{gyu@wm.edu}. Research supported in part by the NSF grant DMS-0852452.}

\maketitle

\begin{abstract}
Let $\mad(G)$ denote the maximum average degree (over all subgraphs) of $G$ and let $\chi_i(G)$ denote the injective chromatic number of $G$.  We prove that if $\Delta\geq 4$ and $\mad(G)<\frac{14}5$, then $\chi_i(G)\leq\Delta+2$.  When $\Delta=3$, we show that $\mad(G)<\frac{36}{13}$ implies $\chi_i(G)\le 5$. In contrast, we give a graph $G$ with $\Delta=3$, $\mad(G)=\frac{36}{13}$, and $\chi_i(G)=6$.  
\end{abstract}

\section{Introduction}
Vertex coloring is one of the central areas of graph theory.
Nearly forty years ago, Karp proved that determining the chromatic number of an arbitrary graph is NP-hard (see~\cite{GJ}).  As a result, much of the work since then has focused on bounding the chromatic number for special classes of graphs and finding efficient algorithms to produce near optimal colorings.  Such results include both the Four Color Theorem \cite{4CT,4CTweb} and the Strong Perfect Graph Theorem \cite{SPGT,SPGT2,SPGT3}.

An \textit{injective coloring} of a graph $G$ is an assignment of colors to the vertices of $G$ so that any two vertices with a common neighbor receive distinct colors.  The \textit{injective chromatic number}, $\chi_i(G)$, is the minimum number of colors needed for an injective coloring.  Injective colorings were introduced by Hahn et al. \cite{hkss}, who
showed applications of the injective chromatic number of the hypercube in the theory of error-correcting codes.  

It's natural to look for relationships between the injective chromatic number, $\chi_i(G)$, and the (standard) chromatic number, $\chi(G)$.
With this goal in mind, we define the \textit{neighboring graph $G^{(2)}$} to be the graph with the same vertex set as $G$ and with its edge set given by $E(G^{(2)})=\{uv:\mbox{vertices $u$ and $v$}$ $\mbox{have a common neighbor in $G$}\}$.  
Note that $\chi_i(G) = \chi(G^{(2)})\le \chi(G^2)$; recall that $V(G^2) = V(G)$ and $uv\in E(G^2)$ if $\rm{dist}(u,v)\le 2$.  

The chromatic number of $G^2$ has important applications in Steganography \cite{FL}, which is the study of hiding messages in other media in a way so that no one, apart from the sender and desired receiver, suspects the presence of a message.  The chromatic number of $G^2$ has also been studied extensively in the case where $G$ is a planar graph~\cite{hm,ms}.  Since injective coloring is a special case of standard vertex coloring, it is natural to ask if determining the injective chromatic number of an arbitrary graph is NP-hard.  Hahn et al.~showed that it is.  So (much like standard coloring), we focus our efforts on bounding the injective chromatic number for special classes of graphs and finding efficient algorithms to produce near optimal colorings.

Since all the neighbors of a common vertex must receive distinct colors,
it is easy to see that $\chi_i(G)\ge \Delta(G)$, where $\Delta(G)$ is the maximum degree of $G$.  When the context is clear, we will simply write $\Delta$.   Many people are interested in graphs with relatively small injective chromatic number (at most $\Delta+c$ for some constant $c$).  One natural candidate for such a family of graphs is planar graphs or, more generally, sparse graphs \cite{hkss, hrw, lst}.  Let  {\it\mad(G)} denote the {\it maximum average degree} (over all subgraphs) of $G$.  We call a class $\mathcal G$ of graphs $\it sparse$ if there exists a constant $k$ such that for all $G\in \mathcal{G}$, we have the inequality $\mad(G) < k$.  An easy application of Euler's formula shows that for every planar graph $G$, we have $\mad(G)<\frac{2g}{g-2}$, where $g$ is the girth of $G$ (the length of its shortest cycle).   

In \cite{dhr}, Doyon, Hahn, and Raspaud showed that for a graph $G$ with maximum degree $\Delta$, the following three results hold:  if $\mad(G) <\frac{14}5$, then $\chi_i(G)\leq\Delta+3$; if $\mad(G) < 3$, then $\chi_i(G)\leq\Delta+4$; and if $\mad(G)<\frac{10}3$, then $\chi_i(G)\leq\Delta+8$.   

In~\cite{cky} the present authors improved some bounds given in~\cite{dhr} and~\cite{lst} in certain cases; specifically, we studied sufficient conditions to imply $\chi_i(G)=\Delta$ and $\chi_i(G)\leq\Delta+1$.  In the current paper, we study conditions such that $\chi_i(G)\leq\Delta+2$.   Our main result is the following theorem.

\begin{theorem}
Let $G$ be a graph with maximum degree $\Delta\geq 4$.  If $\mad(G) < \frac{14}5$, then $\chi_i(G)\leq\Delta+2$.
\label{mainthm}
\end{theorem}
\vspace{-.1in}

%Note that for $\Delta=3$, we have graphs with $\chi_i(G)=6$, even with $\mad(G)=\frac{36}{13}$. 
In contrast,
%For $\Delta=3$, 
the following graph $G$ has $\Delta(G)=3$ and $\chi_i(G)=6$, but has only $\mad(G)=\frac{36}{13}$. 

\begin{example}
Let $G$ be the incidence graph of the Fano Plane.  Observe that $G$ is 3-regular, bipartite, and vertex-transitive.  Consider $H = G-v$, where $v$ is an arbitrary vertex.  To see that $\chi_i(H) = 6$, we only need to note that the vertices in the part of size 6 form a clique in $H^{(2)}$, but the vertices in the part of size 7 do not.  
\smallskip

\label{example1}
\begin{figure}[htb]
\begin{tikzpicture}[line width=1pt, scale = .70] 
\tikzstyle vertexG=[circle,fill=black,minimum size=10pt]
\foreach \a/\b/\c in {1/2/4, 2/3/5, 3/4/6, 4/5/7, 5/6/1, 6/7/2}
{
\draw (\a,-1) node[vertexG] {} -- (\a,1) node[vertexG] {} -- (\b,-1) node[vertexG] {} -- (\a,1) -- (\c,-1)node[vertexG] {};
}
\end{tikzpicture}
\caption{Graph $H$: the incidence graph of the Fano plane, with one vertex deleted.  The deleted vertex was in the top part and was incident to the first, third, and seventh vertices of the bottom part.  If a graph $G$ contains $H$ as a subgraph, then $\chi_i(G)\ge 6$.}
\end{figure}
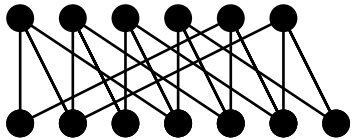
\end{example}

We will show that one cannot construct a graph $G$ with $\Delta(G)=3$, $\chi_i(G)=6$, and $\mad(G)<\frac{36}{13}$. 

\begin{theorem}
If $\Delta=3$ and $\mad(G) < \frac{36}{13}$, then $\chi_i(G)\leq 5$.
\label{Delta3}
\end{theorem}

Hahn, Raspaud, and Wang \cite{hrw} conjectured that every planar graph $G$ with maximum degree $\Delta$ has $\chi_i(G)\le \lceil \frac{3\Delta}{2}\rceil$.   For $\Delta=3$,  the conjecture says that $\chi_i(G)\le 5$. So, by the application of Euler's Formula mentioned above, Theorem~\ref{Delta3} implies that this conjecture is true when the girth of $G$ is at least $8$. 

The rest of the paper is organized as follows.  In Section 2, we introduce reducible configurations, and, as a warmup, we prove Theorem~\ref{Delta3}.  In Section 3,  we prove Theorem~\ref{mainthm}, via the three cases $\Delta\ge 6$, $\Delta=4$, and $\Delta=5$.

\section{Reducible Configurations and Proof of Theorem~\ref{Delta3}}

Before we give the proofs, we will give a brief overview of the discharging method.  First, we introduce our notation.
A $k$-vertex is a vertex of degree $k$; a $k^+$- and a $k^-$-vertex have degree at least and at most $k$, respectively.  A {\it thread} is a path with 2-vertices in its interior and $3^+$-vertices as its endpoints.  A $k$-thread has $k$ interior 2-vertices.  If a $3^+$-vertex $u$ is the endpoint of a thread containing a 2-vertex $v$, then we say that $v$ is a {\it nearby vertex} of $u$ and vice versa.
We write $N_2[u]$ to denote the vertex set consisting of $u$ and its adjacent 2-vertices.

A {\it block} is a maximal 2-connected subgraph.  A {\it list-assignment} is an assignment of (possibly distinct) lists of allowable colors to the vertices of a graph.  A {\it list-coloring} is a proper coloring of the graph such that each vertex $v$ receives a color $c(v)$ from its list $L(v)$.  (The name list-coloring is misleading, since the order of the colors in each ``list'' does not matter at all.)  The language of list-coloring is the natural choice when we are trying to extend a partial coloring of the graph, since the colored vertices may forbid distinct sets of colors on distinct vertices.

All of our proofs rely on the techniques of reducibility and discharging (often these are simply called ``the discharging method'').  
This method was central to the proof of the Four Color Theorem in the 1970s, but has only become widely used in the last 15 years.  Most discharging proofs follow the same pattern.
%We start with a minimal counterexample $G$ to the theorem we are proving.  
We assume, for contradiction, that the theorem we want to prove is false, and we choose graph $G$ to be a minimal counterexample.
We show that $G$ cannot contain certain configurations; we call such a forbidden configuration a {\it reducible configuration}\footnote{A configuration is an induced subgraph, together with prescribed degrees in $G$ for each vertex of the subgraph; a simple example of a forbidden configuration is adjacent 2-vertices.}.
In the {\it discharging phase}, we use a counting argument to show that every supposed minimal counterexample must contain a reducible configuration; this yields a contradiction.

It is useful to observe that a proof by ``minimal counterexample'' (together with the discharging method) is simply a convenient way of rephrasing a proof by induction.  
Thus, when we say, ``by minimality, $G-v$ has an injective coloring with $\Delta+2$ colors'' we are essentially invoking the induction hypothesis.
The advantage of using a minimal counterexample is that we need not specify a base case.  So each reducible configuration represents an inductive step, and our discharging phase shows that in every case at least one of the inductive steps is applicable.  
%We should point out that the subgraph $K$, 
%is, in fact, an additional reducible configuration.  We explain this point further below.

The proofs of Theorem~\ref{Delta3} and Lemma~\ref{main6} are straightforward applications of the discharging method.  However, the proofs of Lemmas~\ref{main4} and~\ref{main5} are more technical.  We use an auxiliary graph $H$; our proofs are based on the idea of 2-alternating cycles, which were introduced by Borodin~\cite{borodin, borodin2} and later extended to 3-alternators by Borodin, Kostochka, and Woodall~\cite{BKW}.
In contrast to earlier reducible configurations, Borodin's 2-alternating cycles (and the subsequent 3-alternators) are infinite classes of reducible configurations, and thus contain reducible configurations that are arbitrarily large.  

This innovation enabled Borodin to get the best known bounds on certain problems related to edge list-coloring.
Our auxiliary graph $H$ allows us to take a similar approach. 
Our subgraph $K$, which arises when analyzing the second discharging phase in the proofs of Lemmas~\ref{main4} and~\ref{main5}, 
is an additional reducible configuration, and it can be arbitrarily large.
All of our proofs yield simple algorithms that produce the desired coloring.
We give more details about such algorithms following the proof of Lemma~\ref{main4}.
%%%% I think the obvious implementation is O(n^2), and a "not-too-hard" improvement gets O(n log n).  I believe it's actually O(n), but don't know how to prove it.

\begin{proof}[Proof of Theorem~\ref{Delta3}:]
Assume that $G$ is a minimal counterexample to  Theorem~\ref{Delta3}, that is, $\mad(G)<\frac{36}{13}$, $\Delta(G)=3$, and $\chi_i(G)>\Delta(G)+2$.  The following four configurations are reducible. 

\medskip

\begin{enumerate}
\item[(RC1)] a 1-vertex.
\item[(RC2)] adjacent 2-vertices.
\item[(RC3)] a 3-vertex adjacent to two 2-vertices. 
\item[(RC4)] adjacent 3-vertices that are each adjacent to a 2-vertex.
\end{enumerate}

\medskip

Now we show that (RC1) -- (RC4) are reducible configurations. In later proofs, when $\Delta > 3$, we will often use the same reducible configurations.  So, here we give proofs that do not use the fact $\Delta=3$, but instead simply assume that every vertex has a list of available colors of size $\Delta+2$. 

 (RC1): Let $v$ be a 1-vertex.  By the minimality of $G$, we can color $G-v$.  Since $v$ has at most $\Delta-1$ colors forbidden, we can extend the coloring to $G$. 
 
 (RC2): Let $u$ and $v$ be adjacent 2-vertices.  By the minimality of $G$, we can color $G\setminus\{u,v\}$.  Again, we can extend the coloring to $G$, since each of $u$ and $v$ has at most $(\Delta-1)+1$ colors forbidden. 
 
 (RC3): Let $u$ be a 3-vertex adjacent to 2-vertices $v$ and $w$, and let $S=\{u,v,w\}$.  By minimality, we can color $G\setminus S$.  Note that $u$ has at most $(\Delta-1)+1+1$ colors forbidden and $v$ and $w$ each have at most $(\Delta-1)+1$ colors forbidden.  Thus, we can extend the coloring to $G$.
 
 (RC4): Let $u_1$ and $u_2$ be adjacent 3-vertices and $v_1$ and $v_2$ be 2-vertices such that $v_i$ is adjacent to $u_i$, and let $S=\{u_1,u_2,v_1,v_2\}$.  By the minimality of $G$, we can color $G\setminus S$.  Note that $u_1$ and $u_2$ each have at most $(\Delta-1)+1+1$ colors forbidden, since the $v_i$'s are uncolored.  After coloring the $u_i$'s, each $v_i$ has at most $(\Delta-1)+1+1$ colors forbidden.  Hence, we can extend the coloring to $G$.

\bigskip

Now we begin the discharging phase; recall that $\Delta(G)=3$.  Our goal is to show that if $G$ has none of the forbidden configurations (RC1) -- (RC4), then $\mad(G)\ge \frac{36}{13}$ (which is a contradiction).
We assign to each vertex $v$ an initial charge $\mu(v)=d(v)$.  We then redistribute this charge by the following two discharging rules:
\smallskip

\begin{enumerate}
\item[(R1)] Each 3-vertex gives charge $\frac3{13}$ to each adjacent 2-vertex.
\item[(R2)] Each 3-vertex gives charge $\frac1{13}$ to each distance-2 2-vertex\\ (unless they lie together on a 4-cycle, in which case the 3-vertex gives $\frac2{13}$ to the 2-vertex).
\end{enumerate}
\medskip

Now we verify that, after discharging, each vertex has charge at least $\frac{36}{13}$. (We write $\mu^*(v)$ to denote the charge at $v$ after applying the discharging rules.)
\smallskip

Recall that $G$ contains no 1-vertex and observe that (RC2) and (RC3) imply that all vertices that are distance at most two from a 2-vertex must be 3-vertices.
By (RC4), no 2-vertex lies on a 3-cycle.  Furthermore, if a 2-vertex $v$ lies on a 4-cycle with a 3-vertex $u$ at distance 2, then $v$ receives $\frac2{13}$ from $u$, rather than just $\frac1{13}$.

Thus, for every 2-vertex $v$, we have $\mu^*(v) = 2 + 2(\frac{3}{13})+4(\frac1{13})=\frac{36}{13}$.

Now we consider 3-vertices.  Note that (RC2), (RC3), and (RC4) together imply that a 3-vertex $v$ cannot have 2-vertices at both distance 1 and 2.  Further, either $v$ has no adjacent 2-vertices and at most three distance-2 2-vertices, or else $v$ has at most one adjacent 2-vertex and no distance-2 2-vertices.  In the first case, we have $\mu^*(v)\geq 3 - 3(\frac1{13}) = \frac{36}{13}$.
In the second case, we have $\mu^*(v)\geq 3 - \frac{3}{13} = \frac{36}{13}$.

Thus, we have $\sum_{v\in V(G)}\frac{36}{13}\le \sum_{v\in V(G)}\mu^*(v) = \sum_{v\in V(G)}\mu(v) = \sum_{v\in V(G)}d(v)$.  Hence, the average degree is at least $\frac{36}{13}$.  This contradiction completes the proof.
\end{proof}

\section{Proof of Theorem~\ref{mainthm}}

To prove Theorem~\ref{mainthm}, we consider separately the cases $\Delta = 4$, $\Delta=5$, and $\Delta\geq 6$.  The proof when $\Delta\geq6$ is similar to the proof of Theorem~\ref{Delta3}, so we consider it first.

\begin{lemma}
If $\Delta\geq 6$ and $\mad(G)<\frac{14}5$, then $\chi_i(G) \leq \Delta + 2$.
\label{main6}
\end{lemma}
\begin{proof}
Suppose the lemma is false; let $G$ be a minimal counterexample.
The following five configurations are reducible.
Proofs for (RC1) -- (RC3) are given in the proof of Theorem~\ref{Delta3}.  Proofs for (RC4) and (RC5) are straightforward and left to the reader.
\bigskip

\begin{enumerate}
\item[(RC1)] a 1-vertex.
\item[(RC2)] adjacent 2-vertices.
\item[(RC3)] a 3-vertex adjacent to two or three 2-vertices.
\item[(RC4)] a 3-vertex adjacent to a 2-vertex and to neighbors $x$ and $y$ with $d(x)+d(y)\leq\Delta+2$.
\item[(RC5)] a 4-vertex adjacent to four 2-vertices such that one of these 2-vertices has other neighbor with degree less than $\Delta$.
\end{enumerate}
\medskip

We use the initial charge $\mu(v) = d(v)$ and the following discharging rules.
\smallskip

\begin{enumerate}
\item[(R1)] each $3^+$-vertex gives charge $\frac25$ to each adjacent 2-vertex.
\item[(R2)] each vertex with degree at least $\ceil{\frac{\Delta+3}2}$ gives charge $\frac25$ to each adjacent 3-vertex or 4-vertex.
\item[(R3)] Suppose that a $4^+$-vertex $v$ is adjacent to $k$ 2-vertices and, after applying rules (R1) and (R2), vertex $v$ has charge $\frac{14}5+l$ (where $l>0$).  For each adjacent 2-vertex $u$, vertex $v$ gives charge $\frac{l}k$ to the other neighbor of $u$.
\end{enumerate}
\medskip

First observe that after applying rules (R1) and (R2), a vertex $v$ has excess charge at least $d(v) - \frac25d(v)-\frac{14}5$; so each vertex $u$ that receives charge from a vertex $v$ by (R3) receives (from $v$) a charge of at least $\frac35-\frac{14}{5d(v)}$.  Note that the application of rule (R3) will never take a vertex from having charge at least $\frac{14}5$ to having charge less than $\frac{14}5$.  Thus, when verifying that each vertex finishes with charge at least $\frac{14}5$, we need not consider charge given away by rule (R3).

Now we verify that all vertices have charge at least $\frac{14}5$.
\smallskip

2-vertex: $\mu^*(v)\geq 2 + 2(\frac25) = \frac{14}5$.

3-vertex: Note that by (RC3) vertex $v$ is adjacent to at most one 2-vertex.
If $v$ is adjacent to zero 2-vertices, then $\mu^*(v) = \mu(v) = 3$.
If $v$ is adjacent to one 2-vertex, then by (RC4) $v$ also has some neighbor with degree at least $\ceil{\frac{\Delta+3}2}$.  So by rule (R2), $\mu^*(v)\geq 3 - \frac25 + \frac25 = 3$.

4-vertex: If $v$ is adjacent to at most three 2-vertices, then $\mu^*(v)\geq 4 - 3(\frac25) = \frac{14}5$.  If $v$ is adjacent to four 2-vertices, then by (RC5), the other neighbor of each adjacent 2-vertex must be a $\Delta$-vertex.  Hence, $\mu^*(v)\geq 4 - 4(\frac25) + 4(\frac35 - \frac{14}{5(6)}) > \frac{14}5$.

$5^+$-vertex: $\mu^*(v)\geq d(v) - \frac25d(v) = \frac35d(v)\geq 3$.
\end{proof}

Now we consider the cases when $\Delta\in \{4,5\}$.  In the proofs thus far, we have extended partial colorings to uncolored vertices simply by counting the number of colors forbidden on an uncolored vertex, and noting that this number is smaller than $\Delta+2$ (the number of colors we can use).  To prove Lemmas~\ref{main4} and \ref{main5}, we need a more subtle argument.  Before, we only cared about how many colors were available at each uncolored vertex.  Now, we also care which colors are available.  We write $L(v)$ to denote the set of colors available at vertex $v$, given a specified partial coloring.
We will need the following two fundamental results on list coloring.

\begin{lemmaB}[Vizing~\cite{vizing}]
\label{vizinglemma}
For a connected graph $G$, let $L$ be a list assignment such that $|L(v)|\geq d(v)$ for all $v$.
(a) If $|L(y)| > d(y)$ for some vertex $y$, then $G$ is $L$-colorable.
(b) If $G$ is 2-connected and the lists are not all identical, then $G$ is $L$-colorable.
\end{lemmaB}

A graph is {\it degree-choosable} if it can be colored from its list assignment $L$ whenever $|L(v)|=d(v)$ for every vertex $v$.

\begin{theoremB}[Erd\H{o}s-Rubin-Taylor ~\cite{ERT}]
\label{west}
A graph $G$ fails to be degree-choosable if and only if every block is a complete graph or an odd cycle.
\end{theoremB}

\begin{lemma}
If $\Delta(G)=4$ and $\mad(G) < \frac{14}5$, then $\chi_i(G)\leq 6$.
\label{main4}
\end{lemma}
\begin{proof}
Suppose the lemma is false; let $G$ be a minimal counterexample.
The following five configurations are reducible.
Proofs for (RC1) -- (RC3) are given in the proof of Theorem~\ref{Delta3}.  Proofs for (RC4) and (RC5) are straightforward and left to the reader.

\begin{enumerate}
\item[(RC1)] a 1-vertex.
\item[(RC2)] adjacent 2-vertices.
\item[(RC3)] a 3-vertex adjacent to two or three 2-vertices.
\item[(RC4)] a 3-vertex adjacent to one 2-vertex and two 3-vertices.
\item[(RC5)] adjacent 3-vertices with each 3-vertex also adjacent to a (possibly distinct) 2-vertex.
\end{enumerate}
\medskip

We again use the initial charge $\mu(v)=d(v)$.
In our first discharging phase, we apply the following two discharging rules:

\begin{enumerate}
\item[(R1.1)] Every $3^+$-vertex gives charge $\frac25$ to each adjacent 2-vertex.
\item[(R1.2)] If $u$ is a 3-vertex adjacent to a 4-vertex $v$ and a 2-vertex, then $v$ gives $\frac15$ to $u$.
\end{enumerate}
\medskip

We consider the charges after the first discharging phase.
\bigskip

2-vertex: 
Configurations (RC1) and (RC2) together imply that each neighbor of a 2-vertex is a $3^+$-vertex.
Thus, $\mu^*(v) = 2 + 2(\frac25) = \frac{14}5$.

3-vertex: If $v$ is adjacent to a 2-vertex, then by (RC4) $v$ is also adjacent to a 4-vertex, so $\mu^*(v)\geq 3 - \frac25 + \frac15 = \frac{14}5$.  Otherwise, $\mu^*(v) = \mu(v) = 3$.

4-vertex: $\mu^*(v) \geq 4 - 4(\frac25) = \frac{12}5$.
\bigskip

Note that every 2-vertex and 3-vertex has charge at least $\frac{14}5$, but 4-vertices can have insufficient charge.  We now construct an auxiliary graph $H$.  
Our aim, in constructing $H$, is to find extra charge to give to the needy 4-vertices.
Graph $H$ will not contain all the vertices of $G$, but $H$ will contain every vertex of $G$ that has charge less than $\frac{14}5$ after the first discharging phase; $H$ will also contain some of the other vertices.

If $H$ is acyclic, then we will show how to complete the discharging argument.  If we cannot complete the discharging argument, then we will use $H$ to show that $G$ contains a reducible configuration.  More specifically, we construct $H$ so that every cycle in $H$ corresponds to an even cycle in $G$ in which each vertex $v$ satisfies $d_{G^{(2)}}(v)\leq 6$; we show if we cannot complete the discharging argument, then one of these even cycles in $G$ is contained in a reducible configuration.

%Each vertex of degree 1\ in $H$ will contribute $\frac15$ to a bank (these will be vertices with excess charge after the initial discharging phase).  Each vertex that has insufficient charge after the initial discharging phase will have degree 3 or 4\ in $H$, and will receive the necessary charge from the bank.  
For convenience, we introduce a subgraph $\widehat{G}^{(2)}$ of $G^{(2)}$.  We form $\widehat{G}^{(2)}$ from $G^{(2)}$ by deleting all 2-vertices of $G$ that have degree at most 5\ in $G^{(2)}$; we can greedily color these vertices after all others.  Hence, it suffices to properly color $\widehat{G}^{(2)}$.  We denote the degree of a vertex $v$ in $\widehat{G}^{(2)}$ by $\widehat{d}(v)$.  We construct $H$ by the three following rules.  We apply rule 3 {\it after} applying rules 1 and 2 everywhere that they are applicable.
\bigskip
\begin{enumerate}
\item[(H1)] If $u$ is a 2-vertex adjacent (in $G$) to vertices $v$ and $w$, then $v,w\in V(H)$ and $vw\in E(H)$.

\item[(H2)] If $u$ is a 3-vertex adjacent (in $G$) to a 3-vertex $v$ and also adjacent to a 2-vertex, then $u,v\in V(H)$ and $uv\in E(H)$.

\item[(H3)] If $v\in V(H)$ and $\widehat{d}(v)\geq 7$, then for each vertex $u$ adjacent to $v$ in $H$ we create a new vertex $v_u$ in $H$ that is adjacent only to vertex $u$; finally, we delete vertex $v$.  (We show below that this rule can only apply when $d_G(v)=4$ and $d_H(v)=2$.)
\end{enumerate}

\begin{figure}[htb]
\scalebox{.8}{
\def\eps{.5}
\def\dot{circle (5pt)}
\begin{tikzpicture}[line width = .02in, scale=1.0]
% "rule" column
\begin{scope}[xshift=-2.5cm]
\draw (0,0) node {\Large{\bf{(H1)}}};
\draw (0,-2cm) node {\Large{\bf{(H2)}}};
\draw (0,-5cm) node {\Large{\bf{(H3)}}};
\end{scope}
% "G" Column heading
\draw (1,1.5) node {\LARGE{$G$}};
% line 1 of G column
\filldraw (-\eps,0) -- (0,0) -- (0,\eps) -- (0,0) \dot -- (1,0) \dot -- (2,0) \dot -- (2,\eps) -- (2,0) -- (2+\eps,0);
\draw[dashed] (0,0) -- (-\eps,\eps) (2,0) -- (2+\eps,\eps);
\draw (0,-.5) node {\Large{$v$}} (1,-.5) node {\Large{$u$}} (2,-.5) node {\Large{$w$}};
\draw (3.6,0) node {\LARGE{$\Longrightarrow$}};
% line 2 of G column
\begin{scope}[yshift=-2cm]
\filldraw (-\eps,0) -- (0,0) -- (0,\eps) -- (0,0) \dot -- (1,0) \dot -- (1,\eps) -- (1,0) -- (2,0) \dot -- (2+\eps,0); 
\draw (0,-.5) node {\Large{$v$}} (1,-.5) node {\Large{$u$}} (2,-.5) node {\Large{$w$}};
\draw (3.6,0) node {\LARGE{$\Longrightarrow$}};
\end{scope}
% line 3 of G column
\begin{scope}[yshift=-5cm]
\filldraw (-.8,\eps) -- (-.8,-\eps) -- (-.8,0) \dot -- (.1,0) \dot -- 
(1,0) \dot -- (1,1) \dot -- (1-\eps,1) -- (1,1) -- (1,1+\eps) -- (1,1) -- (1+\eps,1) -- (1,1) 
(1,0) \dot -- (1,-1) \dot -- (1-\eps,-1) -- (1,-1) -- 
(1+\eps,-1) -- (1,-1) (1,0) 
-- (1.9,0) \dot -- (2.8,0)\dot -- (2.8,\eps) -- (2.8,-\eps); 
\draw[dashed] (1,-1.1-\eps) -- (1,-1);
\draw[dashed] (-.8-\eps,0) -- (-.8,0);
\draw[dashed] (2.8+\eps,0) -- (2.8,0);
\draw (1.3,-.3) node {\Large{$v$}}; 
\draw (-.5,-.3) node {\Large{$u$}}; 
\draw (3.1,-.3) node {\Large{$w$}}; 
\end{scope}
% shift right for "H" column
\begin{scope}[xshift = 5.5cm]
% "H" column heading
\draw (.5,1.5) node {\LARGE{$H$}};
% line 1 of H column
\filldraw (0,0) \dot -- (1,0) \dot;
\draw (0,-.5) node {\Large{$v$}} (1,-.5) node {\Large{$w$}};
% line 2 of H column
\begin{scope}[yshift= -2cm]
\filldraw (0,0) \dot -- (1,0) \dot;
\draw (0,-.5) node {\Large{$v$}} (1,-.5) node {\Large{$u$}};
\end{scope}
\begin{scope}[yshift= -4cm]
\filldraw (-1,0) \dot -- (.5,0) \dot -- (2,0) \dot;
\draw (-1,-.5) node {\Large{$u$}} (.5,-.5) node {\Large{$v$}} (2,-.5) node {\Large{$w$}} (.5,-1.3) node {\LARGE{$\Downarrow$}};
\begin{scope}[yshift=-2cm]
\filldraw (-1,0) \dot -- (0,0) \dot  (1,0) \dot -- (2,0) \dot;
\draw (-1,-.5) node {\Large{$u$}} (0,-.5) node {\Large{$v_u$}} (1,-.5) node {\Large{$v_w$}} (2,-.5) node {\Large{$w$}};
\end{scope}
\end{scope}
\end{scope}
\end{tikzpicture}
}
\caption{Rules (H1), (H2), and (H3), for building the auxiliary graph $H$.}
\end{figure}
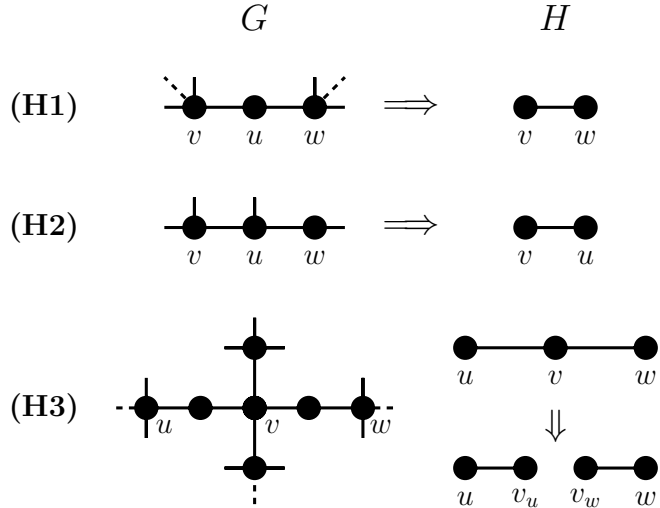

%\subsection{Negative Surplus}~\\
Here we verify the parenthetical claim in (H3).  Note that if $v$ is adjacent in $G$ to three or more 2-vertices, then $\widehat{d}(v)\le 6$.  In contrast if $v$ is adjacent to at most one 2-vertex, then $d_H(v)=1$, so we have nothing to split.  Hence, $v$ must be adjacent to exactly two 2-vertices.  Now the facts that $\widehat{d}(v)\ge 7$ and that $\Delta=4$ imply that $d_G(v)=4$.  
\bigskip

Now we have a second discharging phase, with the following three rules:  
\begin{enumerate}
\item[(R2.1)] Each vertex of degree 1\ in $H$ gives a charge of $\frac15$ to the bank.  
(So, if $v$ was replaced by two vertices, $v_u$ and $v_w$, by rule (H3), then $v$ gives a charge of $\frac25$ to the bank.)
\item[(R2.2)] If a vertex $v$ is in $H$ and in $G$ vertex $v$ is adjacent to three vertices of degree 2 and a vertex of degree 3, then the bank gives $v$ a charge of $\frac15$.  
\item[(R2.3)] If a vertex $v$ is in $H$ and in $G$ vertex $v$ is adjacent to four vertices of degree 2, then the bank gives $v$ a charge of $\frac25$.  
\end{enumerate}

Let $V_{2,2,2,3}$ denote the number of 4-vertices in $G$ that are adjacent to three vertices of degree 2 and one vertex of degree 3; similarly, let $V_{2,2,2,2}$ denote the number of 4-vertices in $G$ that are adjacent to four vertices of degree 2.  Let $\ell$ denote the number of leaves in $H$.  At the end of the second discharging phase, the bank has a charge equal to $\frac15(\ell - V_{2,2,2,3} - 2V_{2,2,2,2})$; we call this charge the {\it surplus}.  We will show that if the surplus is negative, then $G$ contains a reducible configuration and if the surplus is nonnegative, then every vertex of $G$ has charge at least $\frac{14}5$ (which contradicts $\mad(G)<\frac{14}5$).

First, we assume the surplus is negative.
Note that if the surplus is negative, then it must be negative when restricted to some component $J$ of $H$.  Observe that each vertex counted by $V_{2,2,2,3}$ has degree 3\ in $H$ and each counted by $V_{2,2,2,2}$ has degree 4\ in $H$.  Thus, if the surplus is negative when restricted to $J$, then $J$ has average degree greater than 2.  Hence, $J$ contains a cycle $C$ and at least one vertex $u$ counted by either $V_{2,2,2,3}$ or $V_{2,2,2,2}$.  
Recall that $N_2[u]$ is the set consisting of vertex $u$ and all adjacent 2-vertices.  By the minimality of $G$, we have an injective 6-coloring of $G\setminus N_2[u]$ (note that $(G\setminus N_2[u])^{(2)} = G^{(2)}\setminus N_2[u]$); equivalently, this is a proper coloring of $G^{(2)}\setminus N_2[u]$. 

Let $C'$ be the shortest cycle in $G$ that contains all the vertices of $V(C)$ in the order in which they appear in $C$; thus, $V(C')$ contains $V(C)$, as well as some additional 2-vertices and possibly 3-vertices.
Let $K$ be the subgraph of $G$ consisting of $C'$ and a shortest path from $C'$ to $u$ (including $u$); if $u$ lies on $C'$, then we also include in $K$ a 2-vertex that is adjacent to $u$, but that is not responsible for any edge of $C$.
Our proper coloring of $G^{(2)}\setminus N_2[u]$ can naturally be restricted to a proper coloring of $\widehat{G}^{(2)}\setminus N_2[u]$.  We will first modify the coloring of $\widehat{G}^{(2)}\setminus N_2[u]$ to get a proper coloring of $\widehat{G}^{(2)}-V(K)$, then show how to extend this coloring to $\widehat{G}^{(2)}$.
We call these objectives our first and second goals.

If $u$ lies on $C'$, then at most one vertex $w$ of $N_2[u]$ is not in $K$.  Beginning with our coloring of $\widehat{G}^{(2)}\setminus N_2[u]$, we greedily color $w$, then uncolor the vertices of $K$; this yields a coloring of $\widehat{G}^{(2)}-V(K)$.  Thus, if $u$ lies on $C'$, we achieve our first goal.
It's important to notice (and we explain it further in the next paragraph) that $C'$ is an even cycle (in $G$), and hence $V(C')$ forms two disjoint cycles in $\widehat{G}^{(2)}$.  
Thus $K^{(2)}$ consists of two components.
We call the component of $K^{(2)}$ that includes $x$ the \textit{first component} and we call the other component of $K^{(2)}$ the \textit{second component}.  

%The key observation to see that $C'$ is is that, 
To see that $C'$ is an even cycle, note the following.
Due to (RC5), if $C'$ contains an edge created by (H2), then $C'$ contains two successive such edges, yet $C'$ must not contain three successive such edges, since this would force an instance of (RC3) or (RC4).  In contrast (to edges created by (H2)), each edge in $H$ on $C$ that was created by (H1) corresponds to two adjacent edges in $G$ on $C'$.

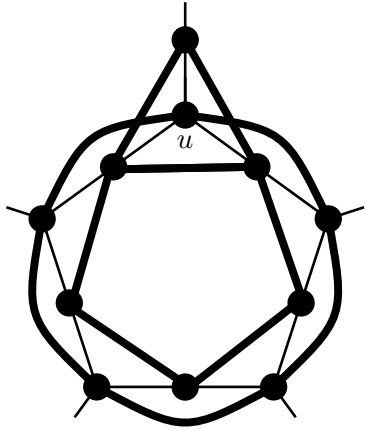
\begin{figure}[htb]
\scalebox{1}{
\def\eps{.5}
\def\dot{circle (5pt)}
\tikzstyle vertexB=[circle,fill=blue,minimum size=12pt]
\tikzstyle vertexBR=[circle,fill=green!95!black,minimum size=12pt]
\tikzstyle vertexG=[circle,fill=black,minimum size=10pt]
\begin{tikzpicture}[line width=1pt, scale = 1.0, color = black] %!45!white] 
\draw (18:2cm) -- (90:2cm) -- (162:2cm) -- (234:2cm) -- (306:2cm) -- cycle;
\draw[line width=3pt] (90:3.0cm) -- (128:1.62cm) -- (54:1.62cm) -- cycle;
\draw[line width=3pt] (54:1.62cm) -- (128:1.62cm) -- (200:1.62cm) -- (272:1.62cm) -- (344:1.62cm) -- cycle;
\draw (90:1.65cm) node {$u$};
\foreach \a in {18, 90, 162, 234, 306}
{
\draw[line width=3pt] (\a:2cm) .. controls (\a+36:2.25cm) .. (\a+72:2cm);
}
%\draw[orange] (18:2cm) .. controls (54:2.3cm) .. (90:2cm);
\draw (90:3.5cm) -- (90:2cm);
\foreach \a in {18, 90, 162, 234, 306}
{
\draw (\a:2cm) -- (\a:2.5cm);
\draw (\a:2cm) node[vertexG] {};
\draw (\a-36:1.62cm) node[vertexG] {};
}
\draw (90:3.0cm) node[vertexG] {};
\end{tikzpicture}
}
\caption{An example of the subgraphs $K$ and $K^{(2)}$, when $u$ lies on $C'$.  The thin lines denote edges of the graph $G$ and the thick lines denote edges of $G^{(2)}$, and, specifically, of $K^{(2)}$.  Vertex $u$ lies in the first component, which is a 5-cycle.  The second component is a ``5-cycle with a hat''.}
\label{K-pic}
\end{figure}

Now we assume instead that $u$ does not lie on $C'$; let $x$ denote the vertex of degree 3\ in $K$.  The main idea in this case is the same as the previous case, but now we have a few more pesky details.  Unfortunately, the present case can occur if $d_G(x)=3$ and all three edges incident to $x$ in $H$ came via rule (H2), rather than (H1)).

A picture in the present case would look similar to the case shown in Figure~\ref{K-pic}.
The only difference is that the first and second components now each have a ``tail'' (path hanging off a single vertex); in the first component the tail attaches at $x$ and in the second component it attaches at the neighbor (in $G$) of $x$ that is not on $C'$ (the ``top of the hat'' in Figure~\ref{K-pic}).
Note that the path from $x$ to $u$ in $G$ is of even length; this is true for the same reason that $C'$ is an even cycle.  Hence, vertex $u$ is in the first component and the vertices in $N_2[u]-u$ are in the second component.  

Starting from our coloring of  $\widehat{G}^{(2)}\setminus N_2[u]$, we uncolor all vertices of the second component. % on the path from $u$ to $C$.  
We now greedily color the uncolored vertices of the second component that are not on $C'$ in order of decreasing distance from $C'$ (as we show in the next paragraph, this uses at most 6 colors).  Finally, we uncolor the vertices of $K$ in the first component. 
This accomplishes our first goal, i.e., it yields a coloring of $\widehat{G}^{(2)}-V(K)$ (except that the vertices in the tail of the first component are already colored).
%The vertices in the second component that remain uncolored are exactly those shown in Figure~\ref{K-pic}---the cycle with a hat.  In contrast, the vertices of the first component that remain uncolored are those on both the cycle and the tail.

%By the minimality of $G$, we have a proper 6-coloring of $\widehat{G}^{(2)}-K$.  
Now we work toward our second goal (extending the above coloring of $\widehat{G}^{(2)}-V(K)$).
Let $L(v)$ denote the list of remaining available colors at each uncolored vertex $v$ of $K$.  
By Lemma~A and Theorem~B, to complete the coloring of $\widehat{G}^{(2)}$, it suffices to show that each component of $K^{(2)}$ either 
contains a vertex $w$ with $|L(w)| > d_{K^{(2)}}(w)$ or contains a block that is neither a clique nor an odd cycle.
Rule (H3) implies that $\widehat{d}(v)\le 6$ for each $v\in V(H)$.
Since each vertex $v$ of $K$ has $\widehat{d}(v)\leq 6$ and we are allowed 6 colors for our injective coloring of $G$, we thus have $|L(v)|\geq d_{K^{(2)}}(v)$ for each vertex $v$.

Since $u$ is counted by either $V_{2,2,2,3}$ or $V_{2,2,2,2}$, we have $\widehat{d}(u)< 6$; hence, we conclude that $d_{K^{(2)}}(u) < |L(u)|$.  Thus, by Lemma~A, we can extend the coloring of $\widehat{G}^{(2)}-V(K)$ to the first component.  Clearly, the second component contains a cycle $E$.  Note that the two neighbors of $x$ that lie on $E$ (and are adjacent to each other in $E$) also have a common neighbor in $K^{(2)}$; hence, the second component contains a block that is not a cycle or a clique.  Thus, by Theorem~$B$, we can extend the coloring of $\widehat{G}^{(2)}-V(K)$ to the second component.  This accomplishes our second goal.  Hence, we have shown that if the surplus is negative, then $\widehat{G}^{(2)}$ contains a reducible configuration.
\smallskip

We now show that if the surplus is nonnegative, then the average degree in $G$ is at least $\frac{14}5$.  We must verify that after each leaf in $H$ gives a charge of $\frac15$ to the bank and each vertex in $H$ counted by $V_{2,2,2,3}$ or $V_{2,2,2,2}$ receives charge from the bank, every vertex has charge at least $\frac{14}5$. Note that if $d_G(v)\le 2$, then $v\notin H$.
To denote the charge at each vertex $v$ after the second discharging phase, we write $\mu^{**}(v)$.  For the analysis that follows, recall that rule (H3) can only apply to a vertex $v$ if $d_G(v)=4$ and $\widehat{d}(v)\ge 7$.

First we consider a vertex $v\in V(H)$ such that $d_G(v)=3$.  Suppose that $d_H(v)=1$.  
Recall that each 2-vertex that is adjacent to $v$ in $G$ corresponds to an edge incident to $v$ in $H$.  Since $d_H(v)=1$, $v$ is adjacent in $G$ to at most one 2-vertex.  Further, if $v$ is adjacent to a 2-vertex, then $v$ is not adjacent to a 3-vertex (since this would imply $d_H(v)\ge 2$).
%If $v$ is adjacent to both a 2-vertex and a 3-vertex, then $\widehat{d}(v)\le 6$, so rule (H3) does not apply to $v$; hence $\mu^{**}(v)=\mu^*(v)\geq\frac{14}5$.
%further, if $v$ is adjacent in $G$ to a 2-vertex, then $v$ is not adjacent in $G$ to any 3-vertex. 
Hence, either $v$ is adjacent in $G$ to one 2-vertex and two 4-vertices or $v$ is not adjacent in $G$ to any 2-vertices.  
In each case, $\mu^*(v)=3$, so $v$ can give charge $\frac15$ to the bank, ending with charge $\mu^{**}(v)=3-\frac15 = \frac{14}5$.  

Now suppose that $d_G(v)=3$ and $d_H(v)\geq 2$.  Either $v$ is adjacent in $G$ to a 2-vertex, a 3-vertex, and a 4-vertex, or $v$ is adjacent in $G$ to at least two 3-vertices and to no 2-vertices.  In the first case $\mu^*(v)=3-1(\frac25)+1(\frac15) = \frac{14}5$, and in the second case $\mu^*(v) = \mu(v) = 3$.
%Note further that in the first case, $d_{G^{(2)}}(v)\leq 6$ and in the second case, $d_{G^{(2)}}(v)\leq 7$.  
%However, in the second case each 3-vertex that is adjacent to $v$ in $H$ is adjacent to a 2-vertex in $G$ that is deleted in $\widehat{G}$; so we have $\widehat{d}(v)\le5$.  Hence, in each case $\widehat{d}(v)\le 6$, so rule (H3) never applies to a vertex $v\in V(H)$ such that $d_G(v)=3$. 
Observe that $v$ does not give away charge in the second discharging phase.
So in both cases, $\mu^{**}(v)=\mu^*(v)\geq\frac{14}5$.
%(Note that rule (H3) never applies to a vertex $v$ with $d_G(v)=3$.)

Now we consider a vertex $v\in V(H)$ such that $d_G(v)=4$.
If vertex $v$ is adjacent in $G$ to at least three 2-vertices, then $\widehat{d}(v)\le 6$, so rule (H3) does not apply to $v$.  
Hence, if $v$ is counted by $V_{2,2,2,3}$, then $\mu^*(v) \geq 4 - 3(\frac25) - 1(\frac15) = \frac{13}5$ and $\mu^{**}(v) = \mu^*(v) + \frac15 = \frac{14}5$; similarly, if $v$ is counted by $V_{2,2,2,2}$, then $\mu^*(v) = 4 - 4(\frac25) = \frac{12}5$ and $\mu^{**}(v) = \mu^*(v) + \frac25 = \frac{14}5$.
If, during the initial discharging phase, $v$ only gave charge to two 2-vertices (and no 3-vertices), then $v$ has sufficient charge to give to the bank if it is split by rule (H3): $\mu^{**}(v)\geq \mu^*(v) - 2(\frac15) = 4 - 2(\frac25) - 2(\frac15) = \frac{14}5$.  Hence, we need only consider the case when during the first discharging phase $v$ gave charge to at most two 2-vertices and at least one 3-vertex.  We examine three subcases.

If $v$ is adjacent in $G$ to two 2-vertices and two 3-vertices, then $\widehat{d}(v)\leq6$, so rule (H3) does not apply to $v$; hence $\mu^{**}(v) = \mu^*(v) = 4 - 2(\frac25) - 2(\frac15) = \frac{14}5$.  If $v$ is adjacent to at most one 2-vertex, then after the initial discharging phase, $\mu^*(v)\geq 4-\frac25-3(\frac15)=3$, so $\mu^{**}(v) = \mu^*(v)-\frac15 = \frac{14}5$.  Finally, suppose that $v$ gave charge to two 2-vertices and one 3-vertex.  If the final neighbor of $v$ is a 4-vertex, then $d_{G^{(2)}}(v)=7$.  However, the 3-vertex adjacent to $v$ is also adjacent to a 2-vertex $u$.  Because $d_{G^{(2)}}(u)\leq 5$, we have $\widehat{d}(v)\leq 6$, so rule (H3) does not apply to $v$.  Hence $\mu^{**}(v) = \mu^*(v) = 4 - 2(\frac25) - 1(\frac15) = 3$.
\end{proof}

The process of converting a discharging proof into an efficient coloring algorithm is well understood (see~\cite{ck} Section 6).  We repeatedly remove reducible configurations, until we reach the empty graph.  We then reconstruct the graph, by adding back the reducible configurations in reverse order, and extending the coloring as we go.  To make this algorithm efficient, at each step we must find a reducible configuration quickly.  If all of our reducible configurations are of bounded size, then we can repeatedly find them in constant amortized time;
 this yields a linear running time.
However, some complications arise here, since the subgraph $K$ may be arbitrarily large.
It is straightforward to overcome these difficulties with an algorithm that runs
in quadratic time.  With more care (and a better choice of data structures), the algorithm can be made to run in $n\log n$ time, where the input graph has $n$ vertices.
\smallskip

The proof of Lemma~\ref{main5} is similar to the proof of Lemma~\ref{main4}, but slightly more complicated.  The additional obstacle we must address in the next proof is verifying that each 5-vertex has sufficient charge.  The additional asset we have is that we are allowed to use 7 colors (rather than the 6 colors allowed in Lemma~\ref{main4}).

\begin{lemma}
If $\Delta(G)=5$ and $\mad(G) < \frac{14}5$, then $\chi_i(G)\leq 7$.
\label{main5}
\end{lemma}
\begin{proof}
Suppose the lemma is false; let $G$ be a minimal counterexample.
The following four configurations are reducible.
Proofs for (RC1) -- (RC3) are given in the proof of Theorem~\ref{Delta3}.  A proof for (RC4) is straightforward and left to the reader.

%\noindent
\begin{enumerate}
\item[(RC1)] a 1-vertex.
\item[(RC2)] adjacent 2-vertices.
\item[(RC3)] a 3-vertex adjacent to two or three 2-vertices.
\item[(RC4)] a 3-vertex adjacent to one 2-vertex and two other vertices $u$ and $v$ with $d(u)+d(v)\leq 7$.
\end{enumerate}

In the first discharging phase, we apply the following three discharging rules:
\smallskip

\begin{enumerate}
\item[(R1.1)] Every $3^+$-vertex gives $\frac25$ to each adjacent 2-vertex.
\item[(R1.2)] If $u$ is a 3-vertex adjacent to two 4-vertices and a 2-vertex, then each adjacent 4-vertex gives $\frac15$ to $u$.
\item[(R1.3)] Every 5-vertex gives $\frac25$ to each adjacent 3-vertex that is adjacent to a 2-vertex and gives $\frac15$ to each adjacent 4-vertex.
\end{enumerate}

We consider the charges after the first discharging phase.
\smallskip

2-vertex: 
Configurations (RC1) and (RC2) together imply that each neighbor of a 2-vertex is a $3^+$-vertex.
Thus, $\mu^*(v) = 2 + 2(\frac25) = \frac{14}5$.

3-vertex: If $v$ is adjacent to a 2-vertex, then by (RC4) $v$ is either adjacent to two 4-vertices or adjacent to a 5-vertex.  In the first case, $\mu^*(v)=3 - \frac25 + 2(\frac15) = 3$.  In the second case, $\mu^*(v)=3 - \frac25 + \frac25 = 3$.  Otherwise, $\mu^*(v) = \mu(v) = 3$.

4-vertex: $\mu^*(v) \geq 4 - 4(\frac25) = \frac{12}5$.

5-vertex: $\mu^*(v)\geq 5 - 5(\frac25) = 3$.
\bigskip

%Note that every 2-vertex charge at least $\frac{14}5$, every 3-vertex and 5-vertex has charge at least 3, but 4-vertices can have insufficient charge.  We now construct an auxiliary graph $H$.  If $H$ is acyclic, then we will show how to finish the discharging argument.  If we cannot complete the discharging argument, then we will use $H$ to show that $G$ contains a reducible configuration.  More specifically, we construct $H$ so that every cycle in $H$ corresponds to a cycle in $G$ in which each vertex $v$ satisfies $d_{G^{(2)}}(v)\leq 7$.
%
%Each vertex of degree 1\ in $H$ will contribute $\frac15$ to a bank (these will be vertices with excess charge after the initial discharging phase).
%Each vertex that has insufficient charge after the initial discharging phase will have degree 3 or 4\ in $H$, and will receive the necessary charge from the bank.  
%
For convenience, we introduce a subgraph $\tilde{G}^{(2)}$ of $G^{(2)}$.  We form $\tilde{G}^{(2)}$ from $G^{(2)}$ by deleting all vertices of $G$ that have degree at most 6\ in $G^{(2)}$; we can greedily color these vertices after all others.  
We denote the degree of a vertex $v$ in $\tilde{G}^{(2)}$ by $\tilde{d}(v)$.  
(Note the subtle difference from the proof of Lemma~\ref{main4}: to form $\widehat{G}^{(2)}$ we only deleted 2-vertices, but now we delete all vertices with $\tilde{d}(v)\leq 6$.  This change is necessary to accomodate the 5-vertices.)
Hence, it suffices to properly color $\tilde{G}^{(2)}$.  
Again we construct an auxiliary graph $H$, to help finish the discharging argument.
We construct $H$ by the two following rules:
%\bigskip
\begin{enumerate}
\item[(H1)] If $u$ is a 2-vertex adjacent to a 4-vertex $v$ and also adjacent to $w$, then $v,w\in V(H)$ and $vw\in E(H)$.

\item[(H2)] If $v\in V(H)$ and 
$\tilde{d}(v)\ge 8$,
then we split $v$ into multiple copies in $H$, as follows.  For each edge $e$ incident to $v$ in $H$, we create a new vertex $v_e$ that is incident only to edge $e$, then we delete the original copy of $v$ in $H$.  
%\item[(S3)] If $e\in E(H)$ is an isolated edge, delete it.  Furthermore, if $uv\in E(H)$, $d_H(v)=1$ and $d_G(u)=3$, then delete edge $uv$ from $H$.
\end{enumerate}
%\bigskip

Now we have a second discharging phase, with the following four rules:  
\begin{enumerate}
\item[(R2.1)] Each vertex of degree 1\ in $H$ gives a charge of $\frac15$ to the bank.  
(So, if $v$ was split into $k$ vertices by rule (H2), then $v$ gives a charge of $\frac{k}5$ to the bank.)
\item[(R2.2)] If a vertex $v$ is in $H$ and in $G$ vertex $v$ is adjacent to three vertices of degree 2 and a vertex of degree 3, then the bank gives $v$ a charge of $\frac15$.  
\item[(R2.3)] If a vertex $v\in V(H)$ and in $G$ vertex $v$ is adjacent to four vertices of degree 2, then the bank gives $v$ a charge of $\frac25$.  
\item[(R2.4)] If a 4-vertex $v$ has charge at least 3 after applying rules (R2.1), (R2.2), (R2.3), then $v$ sends charge $\frac1{15}$ to each 5-vertex $v$ at distance 2 that has a common neighbor $w$ with $u$ such that $d_G(w)=2$.
\end{enumerate}

Let $V_{2,2,2,3}$ denote the number of 4-vertices in $G$ that are adjacent to three vertices of degree 2 and one vertex of degree 3; similarly, let $V_{2,2,2,2}$ denote the number of 4-vertices in $G$ that are adjacent to four vertices of degree 2.  Let $\ell$ denote the number of leaves in $H$.  At the end of the second discharging phase, the bank has a surplus equal to $\frac15(\ell - V_{2,2,2,3} - 2V_{2,2,2,2})$.  We will show that if the surplus is negative, then $G$ contains a reducible configuration and if the surplus is nonnegative, then every vertex of $G$ has charge at least $\frac{14}5$ (which contradicts $\mad(G)<\frac{14}5$).

First, we assume the surplus is negative.
Note that if the surplus is negative, then it must be negative when restricted to
 some component $J$ of $H$.  Observe that each vertex counted by $V_{2,2,2,3}$ has degree 3\ in $H$ and each counted by $V_{2,2,2,2}$ has degree 4\ in $H$.  Thus,
 if the surplus is negative when restricted to $J$, then $J$ has average degree greater than 2.  Hence, $J$ contains a cycle $C$ and at least one vertex $u$ counted by either $V_{2,2,2,3}$ or $V_{2,2,2,2}$.
Recall that $N_2[u]$ is the set consisting of vertex $u$ and all adjacent 2-vertices.  By the minimality of $G$, we have an injective 7-coloring of $G\setminus N_
2[u]$ (note that $(G\setminus N_2[u])^{(2)} = G^{(2)}\setminus N_2[u]$); equivalently, this is a proper coloring of $G^{(2)}\setminus N_2[u]$.

Let $C'$ be the shortest cycle in $G$ that contains all the vertices of $V(C)$ in
 the order in which they appear in $C$; thus, $V(C')$ contains $V(C)$, as well as
 some additional 2-vertices.
Let $K$ be the subgraph of $G$ consisting of $C'$ and a shortest path from $C'$ to $u$ (including $u$); if $u$ lies on $C'$, then we also include in $K$ a 2-vertex that is adjacent to $u$, but that is not responsible for any edge of $C$.
Our proper coloring of $G^{(2)}\setminus N_2[u]$ can naturally be restricted to a proper coloring of $\tilde{G}^{(2)}\setminus N_2[u]$.  We will first modify the coloring of $\tilde{G}^{(2)}\setminus N_2[u]$ to get a proper coloring of $\tilde{G}^{(2)}-V(K)$, then show how to extend this coloring to $\tilde{G}^{(2) }$.
We call these objectives our first and second goals.

%If $u$ lies on $C'$, then at most one vertex $w$ of $N_2[u]$ is not in $K$.  Beginning with our coloring of $\widehat{G}^{(2)}\setminus N_2[u]$, we greedily color $w$, then uncolor the vertices of $K$; this yields a coloring of $\widehat{G}^{(2)}-V(K)$.  Thus, if $u$ lies on $C'$, we achieve our first goal.

If $u$ lies on $C'$, then at most one vertex $w$ of $N_2[u]$ is not in $K$.  Beginning with our coloring of $\tilde{G}^{(2)}\setminus N_2[u]$, we greedily color $w$, then uncolor the vertices of $K$; this yields a coloring of $\tilde{G}^{(2)}-V(K)$.  Thus, if $u$ lies on $C'$, we achieve our first goal.
It's important to notice (and we explain it further in the next paragraph) that $C'$ is an even cycle in $G$, and hence $V(C')$ forms two disjoint cycles in $\tilde{G}^{(2)}$.  
Thus $K^{(2)}$ consists of two components.
Let $x$ denote the vertex of degree 3\ in $K$.  
We call the component of $K^{(2)}$ that includes $x$ the \textit{first component} and we call the other component of $K^{(2)}$ the \textit{second component}.  

%If $u$ lies on $C'$, then at most one vertex $w$ of $N_2[u]$ is not in $K$.  Beginning with our coloring of $\tilde{G}^{(2)}\setminus N_2[u]$, we greedily color $w$, then uncolor the vertices of $K$; this yields a coloring of $G^{(2)}-V(K)$.

%
We now assume that $u$ does not lie on $C'$.
Again $C'$ is an even cycle, and hence $V(C')$ forms two disjoint cycles in $G^{(2)}$.  This observation follows directly from the fact that each edge of $H$ is constructed by rule (H1), and therefore corresponds to two successive edges on $C'$.
%the key observation is that because of (RC5), if $C'$ contains an edge created by (H2), then $C'$ contains two successive such edges, yet $C'$ must not contain three successive such edges, since this would force an instance of (RC4).
%Let $x$ denote the vertex of degree 3\ in $K$.  
%We call the component of $K^{(2)}$ that includes $x$ the \textit{first component} and we call the other component of $K^{(2)}$ the \textit{second component}.  
Note that the path from $x$ to $u$ in $G$ is also of even length; this is true for the same reason that $C'$ is an even cycle.  Hence, vertex $u$ is in the first component and the vertices in $N_2[u]-u$
are in the second component.
Starting from our coloring of  $\tilde{G}^{(2)}\setminus N_2[u]$, we uncolor all vertices of the second component.   Thus, if $u$ does not lie on $C'$, then we achieve our first goal.

Now we work toward our second goal, extending the partial coloring to $\tilde{G}^{(2)}$.
Let $L(v)$ denote the list of remaining available colors at each vertex $v$.  
Rule (H2) implies that $\tilde{d}(v)\le 7$ for each $v\in V(H)$.
Since each vertex $v$ of $K$ has $\tilde{d}(v)\leq 7$ and we are allowed 7 colors for our injective coloring of $G$, we thus have $|L(v)|\geq d_{K^{(2)}}(v)$ for each vertex $v$.  
By Lemma~A and Theorem~B, to complete the coloring of $\tilde{G}^{(2)}$, it suffices to show that each component of $K^{(2)}$ either 
contains a vertex $w$ with $|L(w)| > d_{K^{(2)}}(w)$ or contains a block that is neither a clique nor an odd cycle.

Since $u$ is counted by either $V_{2,2,2,3}$ or $V_{2,2,2,2}$, we have $\tilde{d}(u)< 7$; hence, we conclude $d_{K^{(2)}}(u) < |L(u)|$.  Thus, we can extend the coloring of $\tilde{G}^{(2)}-V(K)$ to the first component.  Clearly, the second component contains a cycle $E$.  Note that the two neighbors of $x$ that lie on $E$ (and are adjacent to each other in $E$) also have a common neighbor in $K^{(2)}$; hence, the second component contains a block that is not a cycle or a clique.  Thus, we can extend the coloring of $\tilde{G}^{(2)}-V(K)$ to the second component.  This achieves our second goal.  Hence, we have shown that if the surplus is negative, then $\tilde{G}^{(2)}$ contains a reducible configuration.
\smallskip

We now show that if the surplus is nonnegative, then the average degree in $G$ is at least $\frac{14}5$.  We must verify that after each leaf in $H$ gives a charge of $\frac15$ to the bank and each vertex in $H$ counted by $V_{2,2,2,3}$ or $V_{2,2,2,2}$ receives charge from the bank, every vertex has charge at least $\frac{14}5$.
To denote the charge at each vertex $v$ after the second discharging phase, we write $\mu^{**}(v)$.

First, we consider a vertex $v\in V(H)$ such that $d_G(v)=3$.  
Note that $d_H(v)\le 1$, since $d_H(v)\ge 2$ would imply that in $G$ vertex $v$ is adjacent to at least two 2-vertices, which contradicts (RC3).
So suppose that $d_H(v)=1$.  
Clearly, $v$ is adjacent to a 2-vertex in $G$.  If $v$ is also adjacent to a 5-vertex, then $\mu^*(v)\geq 3 - \frac25 + \frac25 = 3$.  If $v$ is not adjacent to a 5-vertex, then by (RC3) and (RC4), $v$ must be adjacent to two 4-vertices; hence, $\mu^*(v)\ge 3 - \frac25 + 2(\frac15) = 3$.
In each case, $v$ has charge at least 3 after the initial discharging phase, so $v$ can give charge $\frac15$ to the bank.  

Now, we consider a vertex $v\in V(H)$ such that $d_G(v)=4$.
We must verify that for each such vertex, either $\tilde{d}(v)\leq 6$ or $v$ is able to give sufficient charge to the bank after it is split by rule (H2).  If in $G$ vertex $v$ is adjacent to at least three 2-vertices, then $\tilde{d}(v)\leq 7$.  If in the initial discharging phase, $v$ has only given charge to two 2-vertices (and no 3-vertices), then $v$ has sufficient charge to give to the bank if it is split by rule (H2).  Hence, we need only consider the case when during the first discharging phase $v$ has given charge to at most two 2-vertices and at least one 3-vertex.  
Note, as follows, that rule (R2.4) will never cause the charge of a 4-vertex $v$ to drop below $\frac{14}5$.  If a 4-vertex gives charge by rule (R2.4) to at most three 5-vertices, then $\mu^{**}(v)\geq 3-3(\frac1{15})=\frac{14}5$.  However, if $v$ gives charge by rule (R2.4) to four  5-vertices, then $\mu^{**}(v)=\mu^*(v)-4(\frac1{15})+\frac25 > \frac{14}5$.  Hence, in what follows, we do not consider rule (R2.4).
We examine three subcases.

If $v$ is adjacent in $G$ to two 2-vertices and two 3-vertices, then $\tilde{d}(v)\leq 6$.  If $v$ is adjacent to at most one 2-vertex, then after the initial discharging phase, $v$ has charge at least $4-\frac25-3(\frac15)=3$, so $v$ is able to give charge $\frac15$ to the bank.  Finally, suppose that $v$ has given charge to two 2-vertices and one 3-vertex.  %If the final neighbor of $v$ is a 4-vertex, then $d_{G^{(2)}}(v)=7$.  
Observe that the 3-vertex adjacent to $v$ is also adjacent to a 2-vertex $u$.  Because $d_{G^{(2)}}(u)\leq 6$, we see that $\tilde{d}(v)\leq7$.

Finally, we consider a vertex $v\in H$ such that $d_G(v)=5$.
If $v$ is adjacent in $G$ to at most three 2-vertices and at most four $3^-$-vertices, then $\mu^{**}(v)\ge \mu^*(v)-3(\frac15)\geq 5 - 4(\frac25)-3(\frac15)=\frac{14}5$.
Suppose instead that $v$ is adjacent to five $3^-$-vertices.  If $v$ is adjacent to at least three 2-vertices, then $\tilde{d}(v)\leq 7$, so $v$ is not split by rule (H2).  Thus, $\mu^{**}(v)\geq 5 - 5(\frac25)=3$.
If $v$ is adjacent to five $3^-$-vertices and at least three of them are 3-vertices, then we have the following analysis. 
If $v$ is not split by rule (H2), then $\mu^{**}(v)\geq 5 - 5(\frac25)=3$; hence, we assume that $v$ is split by (H2), which implies that $\tilde{d}(v)\ge 8$.  This inequality  implies that at least three 3-vertices that are adjacent to $v$ are not adjacent to 2-vertices (if such a 3-vertex is adjacent to a 2-vertex $u$, then $d_{G^{(2)}}(u)\leq 6$, so $u$ does not contribute to $\tilde{d}(v)$).  
Hence, these 3-vertices do not receive charge from $v$, so we conclude that $\mu^{**}(v)\geq 5 - 2(\frac25)-2(\frac15)=\frac{19}5 > \frac{14}5$.
%[Recall that a 3-vertex $u$ that is adjacent to $v$ only receives charge from $v$ if $u$ is also adjacent to a 2-vertex. then it gets no charge: good.  If adjacent 3-vert has adj 2-vert u, then $d_{G^{(2)}}(u)\leq 6$, so $u$ does not count toward $\tilde{d}(v)$: good.]

So $v$ must be adjacent to exactly four $3^-$-vertices, and all of these $3^-$-vertices are 2-vertices.  
Consider $d_H(v)$ before we apply rule (H2).  Each edge incident in $H$ to $v$ corresponds to a 2-vertex in $G$ that is  adjacent to $v$ and is also adjacent to a 4-vertex $u$.  If at least two of these 4-vertices have $d_{G^{(2)}}(u)\leq 6$, then  
$\tilde{d}(v)\le 6$, and $v$ is not split by (H2).  
Suppose one such 4-vertex $u$ has $d_{G^{(2)}}(u)\geq 7$. Either $u$ is adjacent to at most two 2-vertices, or $u$ is adjacent to three 2-vertices and one 5-vertex; in both cases, $\mu^*(u)\geq 3$, so $u$ gives charge $\frac1{15}$ to $v$.
Hence, if at least three of these 4-verts have $d_{G^{(2)}}\ge 7$, then $v$ gets charge $\frac1{15}$ from each, so $\mu^{**}(v)\geq 5 - 5(\frac25)-2(\frac15)+3\frac1{15} = \frac{14}5$.
\end{proof}

By combining Lemmas~\ref{main6}, \ref{main4}, and \ref{main5}, we prove Theorem~\ref{mainthm}.
\smallskip

Although we have stated our results only for injective coloring, all of our proofs yield the same bounds for injective list coloring (which is defined analogously).
Our proofs for the reducible configuration are already phrased in terms of list coloring.  Thus, we would only need to use minimality to assume an injective {\it list-}coloring of the subgraphs of our minimal counterexample, rather than simply an injective coloring, as we have done.

\section{Acknowledgements}
Thank you to Farhad Sahhrokhi, Amy Ksir, and an anonymous referee, for comments that improved the paper.  Thank you especially to Beth Kupin, whose very detailed reading of the paper and extensive comments substantially improved the exposition.


\begin{thebibliography}{100}
\bibitem{borodin}
O.V. Borodin, \emph{On the total coloring of planar graphs}, J. reine angew. Math. {\bf 394} (1989), pp. 180--185.

\bibitem{borodin2}
O.V. Borodin, \emph{An extension of Kotzig's theorem and the list edge colouring of planar graphs}, Mat. Zametki {\bf 48} (1990), pp. 22--28. (in Russian)

\bibitem{BKW}
O.V. Borodin, A.V. Kostochka, and D.R. Woodall, \emph{List edge and list total colourings of multigraphs}, J. Comb. Theory B {\bf 71} (1997), pp. 184--204.

\bibitem{SPGT}
M. Chudnovsky, N. Robertson, P. Seymour, and R. Thomas,
\emph{The strong perfect graph theorem},
Annals of Math.	{\bf 164}(1) (2006), pp. 51--229,
\url{http://annals.princeton.edu/annals/2006/164-1/p02.xhtml}

\bibitem{SPGT2}
G. Cornu\'ejols, \emph{The Strong Perfect Graph Theorem}, Optima  {\bf 70} (2003), pp. 2--6,
\url{http://integer.tepper.cmu.edu/webpub/optima.pdf}

\bibitem{ck}
D.W. Cranston and S.-J. Kim, \emph{List-coloring the Square of a Subcubic Graph}, J. of 
Graph Theory {\bf 57} (2008), pp. 65--87.

\bibitem{cky}
D.W. Cranston, S.-J. Kim, and G. Yu, \emph{Injective colorings of sparse graphs},  Discrete Math.  To appear.

\bibitem{dhr}
A. Doyon, G. Hahn, and A. Raspaud, \emph{On the injective chromatic number of sparse graphs}, 
%preprint 2005.
Discrete Math. {\bf 310}(3) (2010), pp. 585--590.

\bibitem{ERT}
P. Erd\H{o}s, A. Rubin, and H. Taylor, \emph{Choosability in graphs}, Congr. Num. {\bf 26} (1979), pp. 125--157.

\bibitem{FL}
J. Fridrich and P. Lisonek, \emph{Grid colorings in Steganography}, IEEE Transactions on Information Theory, {\bf 53} (2007), pp. 1547--1549. 

\bibitem{GJ}
M.R. Garey and D.S. Johnson, \emph{Computers and Intractability: a guide to the theory of NP-completeness}, W.H. Freeman and Company, New York, N.Y., 1979.
\bibitem{hkss}
G. Hahn, J. Kratochv\'il, J. \v{S}ir\'a\v{n}, and D. Sotteau, \emph{On the injective chromatic number of graphs}, Discrete Math. {\bf 256} (2002), pp. 179--192.

\bibitem{hrw}
G. Hahn, A. Raspaud, and W. Wang, \emph{On the injective coloring of $K_4$-minor free graphs}, preprint 2006, 
\url{http://www.labri.fr/perso/lepine/Rapports_internes/RR-140106.ps.gz}

\bibitem{hm}
J. van den Heuvel and S. McGuinness, \emph{Coloring the square of a planar graph}, J. of Graph Theory {\bf 42} (2002), pp. 110--124.

\bibitem{lst}
B. Lu\v{z}ar, R. \v{S}krekovski, and M. Tancer, \emph{Injective colorings of planar graphs with few colors}, 
%preprint 2006.
Discrete Math. {\bf 309}(18) (2009), pp. 5636--5649.

\bibitem{ms}
M. Molloy and M.R. Salavatipour, \emph{A bound on the chromatic number of the square of a planar graph}
J. Combin. Theory B {\bf 94} (2005), pp. 189--213.

\bibitem{4CT}
N. Robertson, D.P. Sanders, P.D. Seymour, and R. Thomas, \emph{The four colour theorem}, J. Combin. Theory B {\bf 70} (1997), pp. 2--44. 

\bibitem{4CTweb}
N. Robertson, D.P. Sanders, P.D. Seymour, and R. Thomas, \emph{The four color theorem}, \url{http://people.math.gatech.edu/~thomas/FC/fourcolor.html}

\bibitem{SPGT3}
P. Seymour, \emph{How the proof of the strong perfect graph conjecture was found},
\url{http://users.encs.concordia.ca/~chvatal/perfect/pds.pdf}

\bibitem{vizing}
V.G. Vizing, \emph{Coloring the vertices of a graph in prescribed colors}, Diskret. Analiz. {\bf 29} (1976), pp. 3--10. (in Russian)
\end{thebibliography}
\end{document}